\documentclass[a4paper,11pt]{article}

\usepackage{amsmath,amsthm,amssymb,mathrsfs,latexsym,amsfonts}
\usepackage{graphicx,psfrag,epsfig}
\usepackage{bbm}
\usepackage[english]{babel}
\usepackage[latin1]{inputenc}
\usepackage{a4wide}

\newtheorem{theorem}{Theorem}

\newtheorem{proposition}{Proposition}

\newtheorem{question}{Question}
\newtheorem{corollary}[theorem]{Corollary}
\newtheorem{definition}{Definition\rm}

\newcounter{paraga}[section]

\newtheorem*{thma}{Theorem 1'}
\newtheorem*{thmb}{Theorem 2'}
\newtheorem*{thmc}{Theorem 3'}

\begin{document}

\def\MP{\,{<\hspace{-.5em}\cdot}\,}
\def\SP{\,{>\hspace{-.3em}\cdot}\,}
\def\PM{\,{\cdot\hspace{-.3em}<}\,}
\def\PS{\,{\cdot\hspace{-.3em}>}\,}
\def\EP{\,{=\hspace{-.2em}\cdot}\,}
\def\PP{\,{+\hspace{-.1em}\cdot}\,}
\def\PE{\,{\cdot\hspace{-.2em}=}\,}
\def\N{\mathbb N}
\def\C{\mathbb C}
\def\Q{\mathbb Q}
\def\R{\mathbb R}
\def\T{\mathbb T}
\def\A{\mathbb A}
\def\Z{\mathbb Z}
\def\demi{\frac{1}{2}}

\begin{titlepage}
\author{Abed Bounemoura~\footnote{Email: abedbou@gmail.com. CNRS - CEREMADE, Université Paris Dauphine \& IMCCE, Observatoire de Paris.}}
\title{\LARGE{\textbf{Generic perturbations of linear integrable Hamiltonian systems}}}
\end{titlepage}

\maketitle

\begin{abstract}
In this paper, we investigate perturbations of linear integrable Hamiltonian systems, with the aim of establishing results in the spirit of the KAM theorem (preservation of invariant tori), the Nekhoroshev theorem (stability of the action variables for a finite but long interval of time) and Arnold diffusion (instability of the action variables). Whether the frequency of the integrable system is resonant or not, it is known that the KAM theorem does not hold true for all perturbations; when the frequency is resonant, it is the Nekhoroshev theorem which does not hold true for all perturbations. Our first result deals with the resonant case: we prove a result of instability for a generic perturbation, which implies that the KAM and the Nekhoroshev theorem do not hold true even for a generic perturbation. The case where the frequency is non-resonant is more subtle. Our second result shows that for a generic perturbation, the KAM theorem holds true. Concerning the Nekhrosohev theorem, it is known that one has stability over an exponentially long (with respect to some function of $\varepsilon^{-1}$) interval of time, and that this cannot be improved for all perturbations. Our third result shows that for a generic perturbation, one has stability for a doubly exponentially long interval of time. The only question left unanswered is whether one has instability for a generic perturbation (necessarily after this very long interval of time).
\end{abstract}

\section{Introduction}\label{s1}

Let $n \geq 2$ be an integer, $\T^n:=\R^n / \Z^n$ and $B \subseteq \R^n$ an open bounded convex subset. Consider a Hamiltonian function $H :\T^n \times B \rightarrow \R$ of the form
\[ H(\theta,I)=h(I)+\varepsilon f(\theta,I), \quad (\theta,I)=(\theta_1,\dots,\theta_n,I_1,\dots,I_n)\in \T^n \times B, \quad \varepsilon \geq 0 \] 
and its associated Hamiltonian system
\begin{equation*}
\begin{cases} 
\dot{\theta}_i(t)=\partial_{I_i} H(\theta(t),I(t))=\partial_{I_i}h(I(t))+ \varepsilon \partial_{I_i} f(\theta(t),I(t)), \\
\dot{I}_i(t)=- \partial_{\theta_i} H(\theta(t),I(t))=- \varepsilon \partial_{\theta_i} f(\theta(t),I(t)) 
\end{cases}
\quad 1 \leq i \leq n. 
\end{equation*} 
For $\varepsilon=0$, all solutions are stable: the action variables are constant, $I(t)=I_0$ for all $t\in\R$, hence their level sets $\{I(t)=I_0\} \times \T^n$ define invariant tori on which the dynamics is quasi-periodic of frequency $\nabla h(I_0)$. A central question is whether solutions remain stable or become unstable when $\varepsilon>0$, and fundamental results on this question are due to Kolmogorov, Arnold and Nekhoroshev.

In \cite{Kol54}, Kolmogorov proved that for a non-degenerate $h$ and for all $f$, the system defined by $H$ still has many invariant tori, provided it is real-analytic and $\varepsilon$ is small enough. What he showed is that among the set of unperturbed invariant tori, there is a subset of positive Lebesgue measure (the complement of which has a measure going to zero when $\varepsilon$ goes to zero) who survives any sufficiently small perturbation, the tori being only slighted deformed. This result has been re-proved and extended in many ways, with important contributions of Arnold and Moser, and this is nowadays called the KAM theorem: the literature on the subject is enormous, we simply refer to \cite{Pos01} and \cite{Sev03} for recent surveys. A consequence of Kolmogorov's result is that for most solutions, the action variables are almost constant for all times, that is 
\[ |I(t)-I_0|\lesssim \sqrt{\varepsilon}, \quad t \in \R. \]
This last inequality is in fact true for all solutions if $n=2$, under a different non-degeneracy condition as proved by Arnold (\cite{Arn63a}, \cite{Arn63b}) with a different version of the KAM theorem.

But this stability of the action variables over an infinite time is not true for all solutions if $n \geq 3$. In \cite{Arn64}, Arnold constructed an example of $h$ (belonging to the class of non-degenerate functions to which the KAM theorem applies) and an example of $f$ such that the system associated to $H$ has a solution for which
\[ |I(t)-I_0|\geq 1, \quad t \geq \tau(\varepsilon)>0, \]
for all $\varepsilon>0$ sufficiently small. He then conjectured that this phenomenon holds true for a generic $h$ and a generic $f$: this is usually called the Arnold diffusion conjecture, and it is one of the main open problem. 

Then, in the seventies, Nekhoroshev proved (\cite{Nek77}, \cite{Nek79}) that for any $n\geq 2$, for a generic $h$ and for all $f$, along all solutions we have
\[ |I(t)-I(0)|\lesssim \varepsilon^b, \quad |t|\lesssim \exp\left(\varepsilon^{-a}\right),  \]  
for some positive exponents $a$ and $b$, provided that $\varepsilon$ is small enough and that the system is real-analytic. So solutions which do not lie on invariant tori are stable not for all time, but during an interval of time which is exponentially long with respect to some power of the inverse of $\varepsilon$. The consequence on Arnold diffusion is that the time of diffusion $\tau(\varepsilon)$ is exponentially large. The class of generic functions in the Nekhoroshev theorem is different from the class of non-degenerate functions in the KAM theorem of  Kolmogorov and Arnold, but their intersection contains, respectively, the class of strictly convex and strictly quasi-convex functions.

It is our purpose in this paper to investigate the fate of these theorems and conjecture in the special case where $h$ is linear, that is
\begin{equation}\label{lin}
h(I)=\alpha \cdot I:=\alpha_1I_1+\cdots \alpha_nI_n, \quad \alpha \in \R^n \setminus \{0\}.
\end{equation}
But first let us give two reasons why it could be interesting to study perturbations of such linear integrable systems. The first one is obvious: locally around any point the main approximation of a non-linear integrable Hamiltonian is given by its linear part, hence perturbations of linear integrable systems describe the local dynamics of perturbations of non-linear integrable systems. The second reason is that it encompasses several interesting problems. For instance, one encounters such systems (or at least very similar systems) when studying the dynamics in the vicinity of a linearly stable (elliptic) equilibrium point, or a Lagrangian quasi-periodic invariant torus or even more generally a linearly stable isotropic reducible quasi-periodic invariant torus. These problems are of real physical interest (we refer to the work of Chirikov, see \cite{CV89} for instance, where these systems are called ``weakly" non-linear as any non-linearity has to come from the perturbation). 

So let us fix $h$ linear, that is $h$ is as in~\eqref{lin} with some vector $\alpha \in \R^n$, and let us come to the question of whether the KAM theorem, the Nekhoroshev theorem and Arnold diffusion holds true. The first point to observe is that $h$, being very degenerate, does not belong to the class of integrable Hamiltonians to which the KAM or Nekhoroshev theorem applies, and Arnold's construction of unstable solutions does not directly apply in this setting. But this does not necessarily means than those results or constructions are not applicable. There are two possibilities for the vector $\alpha \in \R^n$: either it is resonant, that is there exists $k \in \Z^n\setminus \{0\}$ such that the Euclidean scalar product $k\cdot\alpha=0$, or it is non-resonant, and the results will depend drastically on whether $\alpha$ is resonant or not.

So first assume $\alpha$ is resonant. Examples (that is, for specific perturbations) from \cite{Mos60} and \cite{Sev03} show, respectively, that the Nekhoroshev and the KAM theorem cannot be true for all perturbations. Yet it may be possible that they hold true for a generic perturbation. We will show here that this cannot happen. In fact, for a generic perturbation, we will show that there is a fixed open ball $B_*$ in action space such that for $\varepsilon$ sufficiently small and for all initial action $I_0 \in B_*$, one has
\[ |I(t)-I_0|\sim 1, \quad t \sim \varepsilon^{-1}. \]  
This will be the content of Theorem~\ref{thm1}, to which we refer for a precise statement. Hence, in this case, one has diffusion for a generic perturbation and the conjecture of Arnold is true. The fact that the diffusion time is of order $\varepsilon^{-1}$ implies that the Nekhoroshev theorem cannot hold true for a generic perturbation. Moreover, as the result is true for an open ball of initial action whose radius is fixed as $\varepsilon$ goes to zero, the KAM theorem is also violated for a generic perturbation (we refer to the comments after Theorem~\ref{thm1} for a more detailed explanation). 

Next assume that $\alpha$ is non-resonant. The example of \cite{Sev03} applies as well in this case, and the KAM theorem cannot be true for any perturbation. However, we will prove in Theorem~\ref{thm2} that the KAM theorem holds true for a generic perturbation. As far as the Nekhoroshev theorem is concerned, it is already known that for all perturbations and for all solutions, one has the stability estimates
\[ |I(t)-I(0)|\lesssim \Delta_\alpha(\varepsilon^{-1})^{-1}, \quad |t|\lesssim \exp\left(\Delta_\alpha(\varepsilon^{-1})\right),  \] 
where $\Delta_\alpha$ is an increasing function encoding some arithmetic properties of the vector $\alpha$. In the special case where $\alpha$ satisfies a Diophantine condition (that is, $|k\cdot\alpha|\gtrsim |k|^{-\tau}$ for some $\tau\geq n-1$), then $\Delta_\alpha(x) \gtrsim x^{\frac{1}{1+\tau}}$ and one recovers exponential estimates as in the Nekhoroshev theorem. This was proved in \cite{Bou12} in the general case; the Diophantine case was known before, see for instance~\cite{Fas90}. In~\cite{Bou12} it is also proved that there exists a sequence of $\varepsilon_j$-perturbation, $j \in \N$, with $\varepsilon_j$ going to zero as $j$ goes to infinity, such that
\[ |I(t)-I(0)|\sim \Delta_\alpha(\varepsilon_{j}^{-1})^{-1}, \quad |t|\sim \exp\left(\Delta_\alpha(\varepsilon_{j}^{-1})\right),  \] 
and therefore the above stability estimates cannot be improved for an arbitrary perturbation. Yet we will prove that for a generic perturbation, one has much better stability estimates: the time of stability is a double exponential, that is 
\[ |I(t)-I(0)|\lesssim \Delta_\alpha(\varepsilon^{-1})^{-1}, \quad |t|\lesssim \exp\exp\left(\Delta_\alpha(\varepsilon^{-1})\right).  \] 
This will be proved in Theorem~\ref{thm3}. The only remaining question is therefore the following.

\begin{question}\label{quest1}
Assume $\alpha$ is non-resonant. Can one prove the existence of unstable solutions (in the sense of Arnold) for a generic perturbation of $h(I)=\alpha\cdot I$?
\end{question}

Of course, one knows how to construct specific perturbation for which the system has unstable solution: this is done in \cite{Bou12}, but also in \cite{Dou88} where one can find ``less degenerate" perturbations which however cannot be made analytic. For a generic perturbation, according to Theorem~\ref{thm3} the time of instability has to be doubly exponentially large, which is certainly a sign of the difficulty of the problem. As a matter of fact, this question and the Arnold diffusion conjecture are both special case of the following question.

\begin{question}
Consider an arbitrary integrable system defined by a function $h$, with $n\geq 3$ degrees of freedom. Can one prove the existence of unstable solutions for a generic perturbation of $h$?
\end{question}  

Arnold's original conjecture concerns generic integrable Hamiltonians, which certainly do not encompass linear integrable Hamiltonians. Yet the question makes sense for an arbitrary integrable Hamiltonian, and to the best of our knowledge, this is widely open: it is not known if one can construct an example of integrable Hamiltonian (with $n \geq 3$) for which all solutions remain forever stable after a generic perturbation (if it possible, this of course would give a negative answer to the question above).

To conclude, let us just add that we believe that Question~\ref{quest1} is actually more difficult to the conjecture of Arnold: as it will be clear from the arguments below, when dealing with a generic perturbation of a linear (non-resonant) integrable system, one is naturally led to look at a generic perturbation of a non-linear but singular (that is, $\varepsilon$-dependent) integrable system, and this singular case turns out to be more complicated than the regular case.

\section{Main results}\label{s2}

Let us now state more precisely our results. We will assume, without loss of generality, that our vector $\alpha=(\alpha_1,\dots,\alpha_n) \in \R^n$ has unit norm, that is
\[ |\alpha|:=\max_{1 \leq i \leq n}|\alpha_i|=1. \]
Throughout the text, we shall use $|\,\cdot\,|$ to denote the supremum norm for vectors in $\R^n$, but also the induced norm on tensor spaces associated to $\R^n$.

Our perturbation $f_\varepsilon$ will be assumed to depend explicitly on $\varepsilon \geq 0$, but we will require the following two conditions to hold. First, we will assume that $f_\varepsilon$ is uniformly bounded (with respect to some suitable norm, as we will define below) with respect to $\varepsilon$, so that $\varepsilon$ really describes the magnitude of the perturbation. Then, we will also assume that when $\varepsilon$ goes to zero, $f_\varepsilon$ converges (for the appropriate topology) to $f_0=f$. The various generic conditions will be imposed on $f$, and they will continue to hold for $f_\varepsilon$ for all $\varepsilon$ small enough, uniformly in $\varepsilon$.

\subsection{Resonant case}

Our first result deals with a resonant frequency vector $\alpha \in \R^n$. It is no restriction to assume that $\alpha=(0,\tilde{\alpha}) \in \R^d \times \R^{n-d}$ where $\tilde{\alpha} \in \R^{n-d}$ is non-resonant, for some $1 \leq d \leq n-1$. 

The perturbation $f_\varepsilon: \T^{n} \times B \rightarrow \R$ will be assumed to be only sufficiently smooth. More precisely, given an integer $k \geq 3$, we assume that $f_\varepsilon$ belongs to $C^k(\T^n \times \bar{B})$, which is the space of $k$-times differentiable functions on $\T^n \times B$ whose differentials up to order $k$ extend by continuity to $\T^n \times \bar{B}$, where $\bar{B}$ is the closure of $B$. This is a Banach space with the norm
\[ |f_\varepsilon|_{C^k(\T^n \times \bar{B})}:=\max_{0 \leq l \leq k}\left(\sup_{(\theta,I)\in \T^n \times \bar{B}}|\nabla^l f_\varepsilon(\theta,I)|\right). \]
To introduce our generic condition on the perturbation, let us decompose $\theta=(\bar{\theta},\tilde{\theta}) \in \T^d \times \T^{n-d}$ and define the (partial) average $\bar{f}_\varepsilon : \T^{d} \times B \rightarrow \R$ by
\[ \bar{f}_\varepsilon(\bar{\theta},I):=\int_{\T^{n-d}}f_\varepsilon(\bar{\theta},\tilde{\theta},I)d\tilde{\theta}. \]
Let us further define $\bar{f}=\bar{f}_0$. Our condition is then:

\bigskip

$(G1)$ There exists $I_* \in B$ such that the function $\bar{f}_*: \T^d \rightarrow \R$, defined by $\bar{f}_*(\bar{\theta})=\bar{f}(\bar{\theta},I_*)$, is non-constant. Hence there exists $\bar{\theta}_* \in \T^d$ such that $\partial_{\bar{\theta}}\bar{f}(\bar{\theta}_*,I_*) \neq 0$. 

\bigskip

Since the map $f \in C^k(\T^n \times \bar{B}) \mapsto \bar{f} \in C^k(\T^d \times \bar{B})$ is linear, bounded, surjective and hence open, it is easy to check that the subset of functions $f$ satisfying $(G1)$ is both open and dense (in an appropriate topology; for instance it is open in $C^0$ topology and dense in $C^k$ topology, and even in higher regularity when $f$ is more regular). Eventually, let us consider
\begin{equation}\label{Ham1}
\begin{cases}\tag{$H_{res}$}
H(\theta,I)=\alpha\cdot I+\varepsilon f_\varepsilon(\theta,I), \quad \varepsilon\geq 0, \quad f_0=f, \\
\alpha=(0,\tilde{\alpha})\in \R^d\times \R^{n-d}, \quad  \tilde{\alpha} \in \R^{n-d} \; \mathrm{nonresonant}, \quad |\alpha|=1,  \\
|f_\varepsilon|_{C^k(\T^n \times \bar{B})}\leq 1, \quad k\geq 3. 
\end{cases}
\end{equation} 
Before stating our first result, we need to introduce the following notation: given a positive constant $\zeta$, we define
\[ B-\zeta:=\{I \in B \; | \; B(I,\zeta) \subset B\} \]
where $B(I,\zeta)$ is the ball of radius $\zeta$ around $I$. We can now state our first theorem.

\begin{theorem}\label{thm1}
Let $H$ be as in~\eqref{Ham1}, with $f$ satisfying $(G1)$ and
\[ \lim_{\varepsilon \rightarrow 0}|\partial_{\bar{\theta}}\bar{f}_\varepsilon(\bar{\theta}_*,I_*)-\partial_{\bar{\theta}}\bar{f}(\bar{\theta}_*,I_*)|=0.\] 
Then there exist positive constants $\varepsilon_0$, $c_1$ and $c_2$ and an open ball $B_* \subset B-c_1$ which depend on $h$ and $f$ but not on $\varepsilon$ such that if $0<\varepsilon \leq \varepsilon_0$, for any initial action $I_0 \in B_*$, there exists a solution $(\theta(t),I(t))$ of the system associated to $H$ with $I(0)=I_0$ such that
\[ |I(\tau)-I_0|\geq c_1, \quad \tau=c_2\varepsilon^{-1}. \]
\end{theorem}

This is a result of (fast) instability for a generic perturbation, in the case where the frequency is resonant. As the time of instability is linear with respect to $\varepsilon^{-1}$, this result also implies the failure of Nekhoroshev type estimates for a generic perturbation. Finally, observe that unstable solutions cover, in action space, an open ball which is independent of $\varepsilon$: the set of invariant tori, if any, cannot converge to a set of full measure of invariant unperturbed tori, and therefore the KAM theorem cannot hold true for a generic perturbation. 

\bigskip

Note that our condition $(G1)$ is of local nature, and so is our Theorem~\ref{thm1}. A global condition can be formulated as below: 

\bigskip

$(G1')$ For all $I \in \bar{B}$, there exists $\bar{\theta}=\bar{\theta}(I) \in \T^d$ such that $\partial_{\bar{\theta}}\bar{f}(\bar{\theta},I) \neq 0$. 

\bigskip

It is not hard to prove (for instance, using a transversality argument) that this condition $(G1')$ is still satisfied for an open and dense set of functions $f$. Here's our global statement.

\begin{thma}\label{thm1'}
Let $H$ be as in~\eqref{Ham1}, with $f$ satisfying $(G1')$ and 
\[ \lim_{\varepsilon \rightarrow 0}|\bar{f}_\varepsilon-\bar{f}|_{C^1(\T^d \times \bar{B})}=0.\] 
Then there exist positive constants $\varepsilon_0$, $c_1$ and $c_2$ which depend on $h$ and $f$ but not on $\varepsilon$ such that if $0<\varepsilon \leq \varepsilon_0$, for any initial action $I_0 \in B-c_1$, there exists a solution $(\theta(t),I(t))$ of the system associated to $H$ with $I(0)=I_0$ such that
\[ |I(\tau)-I_0|\geq c_1, \quad \tau=c_2\varepsilon^{-1}. \]
\end{thma}

The proof of Theorem 1' will be just a slight modification of the proof of Theorem~\ref{thm1}.

\subsection{Non-resonant case}

Our next results deal with a non-resonant frequency $\alpha \in \R^n$. Let us define the functions $\Psi_{\alpha}$ by
\begin{equation*}
\Psi_{\alpha}(Q)=\max\left\{|k\cdot\alpha|^{-1} \; | \; k\in\Z^{n}, \; 0<|k|\leq Q \right\}, \quad Q \geq 1
\end{equation*} 
and $\Delta_{\alpha}$ by
\begin{equation*}
\Delta_{\alpha}(x)=\sup\{Q \geq 1\; | \; Q\Psi_{\alpha}(Q)\leq x\}, \quad x \geq 1.
\end{equation*} 
The vector $\alpha$ is said to be $(\gamma,\tau)$-Diophantine for some $\gamma>0$ and $\tau \geq n-1$ if $\Psi_{\alpha}(Q)\leq \gamma^{-1} Q^\tau$: in this case, one has $\Delta_{\alpha}(x) \geq (\gamma x)^{\frac{1}{1+\tau}}$. 

The perturbation $f_\varepsilon: \T^{n} \times B \rightarrow \R$ will be assumed, for simplicity, to be real-analytic but this can be weakened (see the comments after Theorem~\ref{thm2} and Theorem~\ref{thm3}). Given two parameters $r>0$ and $s>0$, let us define the complex neighborhoods
\[  V_s\T^n:=\{\theta \in \C^n / \Z^n \; | \; \max_{1 \leq i \leq n}\{|\mathrm{Im}(\theta_i)|\}<s\}, \quad V_rB:=\{I \in \C^n \; | \; \inf_{I' \in B}|I-I'|<r\}, \]
and for simplicity, let us denote
\[ V_{s,r}:=V_s\T^n \times V_rB.  \]
Our assumption is that $f_\varepsilon$ extends to a bounded holomorphic function on $V_{s,r}$, and let us define
\[ |f_\varepsilon|_{s,r}:=\sup_{(\theta,I) \in V_{s,r}}|f_\varepsilon(\theta,I)|. \]
Let us now define the (full) average of $f_\varepsilon$, that we still denote by $\bar{f}_\varepsilon$, by setting
\begin{equation}\label{ave}
\bar{f}_\varepsilon(I):=\int_{\T^{n}}f_\varepsilon(\theta,I)d\theta.
\end{equation}
As before, let $\bar{f}=\bar{f}_0$ and we consider the following condition:

\bigskip

$(G2)$ There exists $I_* \in B$ such that the Hessian matrix $\nabla^2 \bar{f}(I_*)$ is non-singular.

\bigskip

The set of functions $f$ which satisfy $(G2)$ is obviously open and dense. We are finally led to consider the following Hamiltonian:
\begin{equation}\label{Ham2}
\begin{cases}\tag{$H_{nonres}$}
H(\theta,I)=\alpha\cdot I+\varepsilon f_\varepsilon(\theta,I), \quad \varepsilon \geq 0, \quad f_0=f, \\
\alpha \in \R^{n} \; \mathrm{nonresonant}, \quad |\alpha|=1,  \\
|f_\varepsilon|_{s,r}\leq 1. 
\end{cases}
\end{equation} 

We can now state our second result.

\begin{theorem}\label{thm2}
Let $H$ be as in~\eqref{Ham2}, with $f$ satisfying $(G2)$ and 
\[ \lim_{\varepsilon \rightarrow 0}|\nabla^2 \bar{f}_\varepsilon(I_*)-\nabla^2 \bar{f}(I_*)|=0.\] 
Then there exist positive constant $\bar{\varepsilon}_0$ $c$, $\bar{c}$, $\bar{C}$ and an open ball $B_* \subset B$ which depend only on $h$ and $f$ but not on $\varepsilon$, such that if $0<\varepsilon \leq \bar{\varepsilon}_0$, there exists a set $\mathcal{K}_\varepsilon \subset \T^n \times B_*$, which consists of real-analytic Lagrangian quasi-periodic tori invariant by the Hamiltonian flow of $H$. Moreover, we have the measure estimate
\[ \mathrm{Leb}(\T^n \times B_* \setminus \mathcal{K}_\varepsilon) \leq \bar{C}\exp\left(-\bar{c}\Delta_{\alpha}\left(c\varepsilon^{-1}\right)\right)\mathrm{Leb}(\T^n \times B_*) \]
where $\mathrm{Leb}$ denotes Lebesgue measure.
\end{theorem}

The conclusions of this statement are similar to those of the classical KAM theorem (see for instance \cite{Pos82} or \cite{Pos01}): the proof will actually consist of using the condition $(G2)$ on the perturbation to remove the degeneracy on the integrable part. 

Observe that the measure estimate of the complement of the set of invariant tori depends on $\alpha$, as it is exponentially small with respect to $\Delta_{\alpha}(\varepsilon^{-1})$; when $\alpha$ is $(\gamma,\tau)$-Diophantine, the estimate is of order $\exp\left(-(\gamma/\varepsilon)^{\frac{1}{1+\tau}}\right)$. The analyticity assumption is in fact not necessary; one may assume that the system is only Gevrey regular (see \cite{Pop04}) or sufficiently smooth. In those cases, the statement is essentially the same, with the obvious modifications (concerning the regularity of each torus and the measure estimate) so we will not pursue this further.

\bigskip

Observe also, as before, that our condition $(G2)$ is local, so Theorem~\ref{thm2} only applies to some (possibly very small) open ball $B_*$. To have a global statement, one can consider the following condition:

\bigskip

$(G2')$ $\nabla \bar{f} : \bar{B} \rightarrow \R^n$ is a diffeomorphism onto its image.

\bigskip

Notice, however, that this condition $(G2')$ is no longer generic (one has the same issue for the classical KAM theorem on a fixed domain). Under such a condition, we have the following global statement.

\begin{thmb}\label{thm2'}
Let $H$ be as in~\eqref{Ham2}, with $f$ satisfying $(G2')$ and 
\[ \lim_{\varepsilon \rightarrow 0}|f_\varepsilon-f|_{C^2(\bar{B})}=0.\] 
Then there exist positive constant $\bar{\varepsilon}_0$ $c$, $\bar{c}$ and $\bar{C}$ which depend only on $h$ and $f$ but not on $\varepsilon$, such that if $0<\varepsilon \leq \bar{\varepsilon}_0$, there exists a set $\mathcal{K}_\varepsilon \subset \T^n \times B$, which consists of real-analytic Lagrangian quasi-periodic tori invariant by the Hamiltonian flow of $H$. Moreover, we have the measure estimate
\[ \mathrm{Leb}(\T^n \times B \setminus \mathcal{K}_\varepsilon) \leq \bar{C}\exp\left(-\bar{c}\Delta_{\alpha}\left(c\varepsilon^{-1}\right)\right)\mathrm{Leb}(\T^n \times B) \]
where $\mathrm{Leb}$ denotes Lebesgue measure.
\end{thmb}

As before, the proof of Theorem 2' will be easily obtained from the proof of Theorem~\ref{thm2}.

\bigskip

To state our last result, we need to introduce several notations and a definition. Given an integer $m \geq 2$, let us denote by $P(n,m)$ the space of polynomials with real coefficients of degree $m$ in $n$ variables, and $P_2(n,m)\subset P(n,m)$ the subspace of polynomials with vanishing homogeneous terms of degree zero and one. The following definition was introduced in~\cite{BFN15}, following the work of Nekhoroshev (\cite{Nek73}).  

\begin{definition}\label{stabsteep} 
Given positive constants $\rho$, $C$ and $\delta$, a polynomial $P_0 \in P_2(n,m)$ is called $(\rho,C,\delta)$-stably steep if for any $P \in P_2(n,m)$ such that $|P-P_0|<\rho$, any integer $l\in [1,n-1]$, and any vector subspace $\Lambda \subseteq \R^n$ of dimension $l$, letting $P_\Lambda$ be the restriction of $P$ to $\Lambda$, the inequality
\[ \max_{0 \leq \eta \leq \zeta}\;\min_{|x|=\eta, \; x \in \Lambda}|\nabla P_\Lambda(x)|>C\zeta^{m-1} \]
holds true for all $0 < \zeta \leq\delta$. A polynomial $P_0 \in P_2(n,m)$ is called stably steep if it is $(\rho,C,\delta)$-stably steep for some positive constants $\rho$, $C$ and $\delta$.  
\end{definition} 

Recall that the average perturbation $\bar{f}_\varepsilon : B \rightarrow \R$ was defined in~\eqref{ave}, and that $\bar{f}=\bar{f}_0$. Given $I \in B$ and $m \geq 2$, we let $T_2^m\bar{f}(I) \in P_2(n,m)$ be the polynomial obtained by selecting the homogeneous terms of degree at least $2$ and at most $m$ from the Taylor expansion of $\bar{f}$ at $I$. Formally, we have
\[ T_2^m\bar{f}(I)X=\sum_{l=2}^{m}(l!)^{-1}\nabla^l \bar{f}(I).(X)^l, \quad X=(X_1,\dots,X_n).\]
Let us further define $m_n:=[n^2/2+2]$, where $[\,\cdot\,]$ denotes the integer part, so that we can finally state the condition we shall require on $\bar{f}$:

\bigskip

$(G3)$ There exists $I_* \in B$ and an integer $m \geq m_n$ such that the polynomial $T_2^{m}\bar{f}(I_*) \in P_2(n,m)$ is stably steep.

\bigskip

It was proved in~\cite{BFN15} (using arguments from~\cite{Nek73}) that the set of stably steep polynomials is, for $m \geq m_n$, generic in the sense that its complement is a closed semi-alegebraic subset of positive codimension. The set of functions $f$ which satisfy $(G3)$ is therefore open and dense. We can now state our third result.

\begin{theorem}\label{thm3}
Let $H$ be as in~\eqref{Ham2}, with $f$ satisfying $(G3)$ and 
\[ \lim_{\varepsilon \rightarrow 0}\left(\max_{2 \leq l \leq m}|\nabla^l \bar{f}_\varepsilon(I_*)-\nabla^l \bar{f}(I_*)|\right)=0.\] 
Then there exist positive constant $\tilde{\varepsilon}_0$, $c$, $\tilde{c}$, $\tilde{C}$ and an open ball $B_* \subset B$ which depend only on $h$ and $f$ but not on $\varepsilon$, such that if $0<\varepsilon \leq \tilde{\varepsilon}_0$, for any $I_0 \in B_*$ and any solution $(\theta(t),I(t))$ of the system associated to $H$ with $I(0)=I_0$, we have
\[ |I(t)-I_0|\leq \tilde{C}\Delta_{\alpha}(c\varepsilon^{-1})^{-1}, \quad |t|\leq \varepsilon^{-1}\exp\exp\left(-\tilde{c}\Delta_{\alpha}(c\varepsilon^{-1})\right). \]
\end{theorem}

Exactly as in Theorem~\ref{thm2}, the real-analyticity assumption is not necessary and can be weakened: if $H$ is only assumed to be $\alpha$-Gevrey smooth, for $\alpha\geq 1$, then one obtains the following estimates:
\[ |I(t)-I_0|\leq \tilde{C}\Delta_{\alpha}(c\varepsilon^{-1})^{-1}, \quad |t|\leq \varepsilon^{-1}\exp\exp\left(-\tilde{c}\Delta_{\alpha}(c\varepsilon^{-1})^{\frac{1}{\alpha}}\right), \]
which generalizes the estimates of Theorem~\ref{thm3} (as the case $\alpha=1$ corresponds to real-analytic Hamiltonians). Once again, we can have a global statement under the following global condition:

\bigskip

$(G3')$ There exists an integer $m \geq m_n$ such that for all $I \in \bar{B}$, the polynomial $T_2^{m}\bar{f}(I) \in P_2(n,m)$ is stably steep.

\bigskip

As we already said, the set of stably steep polynomials with $m \geq m_n$ has a complement with positive codimension, but more is true: when $m$ is arbitrary large, the codimension of the complement gets arbitrary large. Using a transversality argument, one can then infer that condition $(G3')$ is still satisfied for an open and dense set of functions $f$. Here's our global statement.

\begin{thmc}\label{thm3'}
Let $H$ be as in~\eqref{Ham2}, with $f$ satisfying $(G3')$ and 
\[ \lim_{\varepsilon \rightarrow 0}|f_\varepsilon-f|_{C^m(\bar{B})}=0.\] 
Then there exist positive constant $\tilde{\varepsilon}_0$, $c$, $\tilde{c}$ and $\tilde{C}$ which depend only on $h$ and $f$ but not on $\varepsilon$, such that if $0<\varepsilon \leq \tilde{\varepsilon}_0$, for any solution $(\theta(t),I(t))$ of the system associated to $H$ with 
\[ I(0)=I_0 \in B-\tilde{C}\Delta_{\alpha}(c\varepsilon^{-1})^{-1}\] 
we have
\[ |I(t)-I_0|\leq \tilde{C}\Delta_{\alpha}(c\varepsilon^{-1})^{-1}, \quad |t|\leq \varepsilon^{-1}\exp\exp\left(-\tilde{c}\Delta_{\alpha}(c\varepsilon^{-1})\right). \]
\end{thmc}

Again, the proof of Theorem 3' will be the same as the proof of Theorem~\ref{thm3}.

\subsection{Applications}

The techniques used to prove Theorem~\ref{thm1}, Theorem~\ref{thm2} and Theorem~\ref{thm3} have been already used in the study of the dynamics in a neighborhood of an invariant quasi-periodic torus. In this case, the Hamiltonian is a small perturbation of a linear integrable Hamiltonian system, so our results are in principle applicable; however, the perturbation being quite specific, our generic conditions $(G1)$, $(G2)$ or $(G3)$ cannot be always satisfied. To be more precise, one considers here a Hamiltonian of the form
\begin{equation}\label{Ham}\tag{$*$}
H(\theta,I)=\alpha \cdot I + A(\theta)I\cdot I+R(\theta,I), \quad (\theta,I)\in \T^n \times U  
\end{equation}
where $U$ is an open neighborhood of the origin in $\R^n$, $A: \T^n \rightarrow \mathrm{Sym}(n,\R)$ is a map taking values in the space of real symmetric matrices of size $n$, and $R(\theta,I)=O_3(I)$ is of order at least $3$ in $I$. The set $\mathcal{T}_\alpha:=\T^n \times \{I=0\}$ is invariant by the Hamiltonian flow of $H$, it is a Lagrangian quasi-periodic torus with frequency $\alpha$. On the domain $\T^n \times B_\varepsilon$, where $B_\varepsilon$ is the $\varepsilon$-ball around the origin in $\R^n$, the Hamiltonian in~\eqref{Ham} is easily seen to be equivalent to 
\begin{equation*}
H'(\theta,I)=\alpha \cdot I + \varepsilon A(\theta)I\cdot I+\varepsilon^2 R(\theta,I), \quad (\theta,I)\in \T^n \times B_1
\end{equation*} 
and this Hamiltonian $H'$ is of the form~\eqref{Ham1} or~\eqref{Ham2} (depending whether $\alpha$ is resonant or not, and on the regularity of $H$) with
\[  f_\varepsilon=A(\theta)I\cdot I+\varepsilon R(\theta,I), \quad f=f_0=A(\theta)I\cdot I. \] 
It is not hard to see that Theorem~\ref{thm1} cannot be used in a neighborhood of a resonant invariant quasi-periodic torus to prove topological instability, as $(G1)$ do not hold at $I_*=0$. Yet under different generic conditions and using techniques similar to those that will be used to prove Theorem~\ref{thm1}, one can still prove that generically such a torus possesses some weaker instability properties, as shown in~\cite{Bou15mrl}.

Concerning Theorem~\ref{thm2}, the situation is much better as it is directly applicable: the condition $(G2)$ amounts to say that the matrix $A_0:=\int_{\T^n}A(\theta)d\theta$ is non-singular. Hence one can show that a non-resonant invariant quasi-periodic torus is always KAM stable, without requiring any Diophantine condition on the frequency vector. This application has already been worked out in~\cite{Bou14a}.

Finally, Theorem~\ref{thm3} is not sufficient to prove that generically, a non-resonant invariant quasi-periodic torus is doubly exponentially stable. Indeed, in this case $f=f_0$ is a quadratic function of $I$ and therefore $(G3)$ cannot be checked. The result does apply if $(G3)$ is replaced by the following stronger condition:

\bigskip

$(G4)$ There exists $I_*$ such that the restriction of $\nabla^2 \bar{f}(I_*)$ to the orthogonal complement of $\alpha \in \R^n$ is sign-definite.

\bigskip

Theorem~\ref{thm3} do apply if $(G3)$ is replaced by $(G4)$, and its proof can be greatly simplified in this case (this is because, for $\varepsilon$ small enough, $(G4)$ implies that $\bar{f}_\varepsilon$ is quasi-convex in a neighborhood of $I^*$, while $(G4)$ implies that $\bar{f}_\varepsilon$ is steep in a neighborhood of $I^*$, a much more general condition that will be recalled in~\S\ref{s6}). However, this condition $(G4)$ is no longer generic. Using Theorem~\ref{thm3} (with the condition $(G3)$ replaced by $(G4)$), we obtain the following corollary.

\begin{corollary}
Let $H$ be as in~\eqref{Ham}, assume that $H$ is real-analytic, $\alpha$ is non-resonant and $A_0$ is sign-definite when restricted to the orthogonal complement of $\alpha \in \R^n$. Then there exist positive constants $c$ and $\tilde{c}$ such that for $\varepsilon>0$ small enough and for any solution $(\theta(t),I(t))$ of the system associated to $H$ with $|I(0)|\leq \varepsilon$, we have
\[ |I(t)|\leq 2\varepsilon, \quad |t|\leq \varepsilon^{-1}\exp\exp\left(-\tilde{c}\Delta_{\alpha}(c\varepsilon^{-1})\right). \]
\end{corollary}

This last result gives a slight generalization of a result of Morbidelli and Giorgilli (see \cite{MG95}), in which the above statement is proved assuming $\alpha$ to be Diophantine; note that their proof do not extend to an arbitrary non-resonant vector $\alpha$ as they use Birkhoff normal forms, which in general do not exist unless $\alpha$ is Diophantine. On the other hand, assuming $\alpha$ to be Diophantine (and using Birkhoff normal forms), it is proved in \cite{BFN15b} that generically, the invariant torus is doubly exponentially stable (an analogous result, in the context of elliptic equilibrium points, is contained in \cite{BFN15}). We do not know if this last result can be obtained without the assumption that $\alpha$ is Diophantine.

\section{Normal forms}\label{s3}

The main ingredients to prove Theorem~\ref{thm1}, Theorem~\ref{thm2} and Theorem~\ref{thm3} are normal form results that will be recalled in this section.  

First we assume that $\alpha \in \R^n$ is resonant, more precisely $\alpha=(0,\tilde{\alpha}) \in \R^d \times \R^{n-d}$ with $1 \leq d \leq n-1$ and $\tilde{\alpha}\in \R^{n-d}$ non-resonant. We can also define $\Psi_{\tilde{\alpha}}$ by
\begin{equation*}
\Psi_{\tilde{\alpha}}(Q)=\max\left\{|k\cdot\tilde{\alpha}|^{-1} \; | \; k\in\Z^{n-d}, \; 0<|k|\leq Q \right\}, \quad Q \geq 1
\end{equation*} 
and $\Delta_{\tilde{\alpha}}$ by
\begin{equation*}
\Delta_{\tilde{\alpha}}(x)=\sup\{Q \geq 1\; | \; Q\Psi_{\tilde{\alpha}}(Q)\leq x\}, \quad x \geq 1.
\end{equation*} 
The proposition below will be used to prove Theorem~\ref{thm1}. 

\begin{proposition}\label{normal1}
Let $H$ be as in~\eqref{Ham1}. Then there exist constants $c:=c(n)$, $C:=C(n,k)$ and $\mu_0:=\mu_0(n,k)$ such that if
\begin{equation}\label{mu}
\mu_{\tilde{\alpha}}(\varepsilon):=\Delta_{\tilde{\alpha}}\left(c\varepsilon^{-1}\right)^{-1} \leq \mu_0,
\end{equation}
then, setting $B'=B-C\mu_{\tilde{\alpha}}(\varepsilon)$, there exist a symplectic embedding $\Phi :\T^n \times \bar{B}' \rightarrow \T^n \times \bar{B}$ of class $C^{k-1}$ such that
\[ H \circ \Phi(\theta,I)=\alpha \cdot I+\varepsilon \bar{f}_\varepsilon(\bar{\theta},I)+ \varepsilon f_\varepsilon^+(\theta,I) \]
with the estimates 
\begin{equation}\label{estimates}
\begin{cases}
|f_\varepsilon^+|_{C^{k-1}(\T^n \times \bar{B}')}\leq C \mu_{\tilde{\alpha}}(\varepsilon), \\ 
|\Phi-\mathrm{Id}|_{C^{k-1}(\T^n \times \bar{B}')} \leq C \mu_{\tilde{\alpha}}(\varepsilon).
\end{cases}
\end{equation}
\end{proposition}

This statement is a special case of Theorem 1.1 in~\cite{Bou13b}, to which we refer for a proof.

\bigskip

In the case where $\alpha \in \R^n$ is non-resonant and the Hamiltonian is real-analytic, we have the following proposition, which will be used to prove both Theorem~\ref{thm2} and Theorem~\ref{thm3}. 

\begin{proposition}\label{normal2}
Let $H$ be as in~\eqref{Ham2}. Then there exist constants $c:=c(n)$, $c':=c'(n,s)$, $C:=C(n,s,r)$ and $\mu_0:=\mu_0(n,s,r)$ such that if
\begin{equation}\label{mu2}
\mu_{\alpha}(\varepsilon):=\Delta_{\alpha}\left(c\varepsilon^{-1}\right)^{-1} \leq \mu_0,
\end{equation}
then there exist a real-analytic symplectic embedding $\Phi :V_{s/2,r/2} \rightarrow V_{s,r}$ such that
\[ H \circ \Phi(\theta,I)=\alpha \cdot I+\varepsilon \bar{f}_\varepsilon(I)+\varepsilon g_\varepsilon(I) +\varepsilon f_\varepsilon^+(\theta,I) \]
with the estimates 
\begin{equation}\label{estimates2}
\begin{cases}
|g_\varepsilon|_{s/2,r/2}\leq C \mu_{\alpha}(\varepsilon), \\ 
|f_\varepsilon^+|_{s/2,r/2}\leq C \exp\left(-c'\mu_{\alpha}(\varepsilon)^{-1}\right),  \\
|\Phi-\mathrm{Id}|_{s/2,r/2} \leq C\mu_{\alpha}(\varepsilon).
\end{cases}
\end{equation}
\end{proposition}

As before, this statement is a special case of Theorem 1.1 in~\cite{Bou13a}, to which we refer for a proof.

\section{Proofs of Theorem~\ref{thm1} and Theorem 1'}\label{s4}

The proof of Theorem~\ref{thm1} will follow easily from Proposition~\ref{normal1}, using arguments similar to those used in~\cite{BK14} and~\cite{BK15}.
 
Proposition~\ref{normal1} provides us with a (symplectic, close to the identity) change of coordinates which reveals that the dominant part of the Hamiltonian is given by $h+\varepsilon\bar{f}_\varepsilon$, where $\bar{f}_\varepsilon$ is the (partial) average of $f_\varepsilon$. The system associated to this dominant part is no longer integrable, because of the condition $(G1)$ it has unstable solutions and an elementary argument will show that the same is true for the system associated to $H$. 

\begin{proof}[Proof of Theorem~\ref{thm1}]
Recall that we are considering $H$ as in~\eqref{Ham1}, with $f$ satisfying $(G1)$, which is the existence of $I_* \in B$ and $\bar{\theta}_* \in \T^d$ such that $\partial_{\bar{\theta}}\bar{f}(\bar{\theta}_*,I_*) \neq 0$. Recall also that
\begin{equation}\label{speed1}
\lim_{\varepsilon \rightarrow 0}|\partial_{\bar{\theta}}\bar{f}_\varepsilon(\bar{\theta}_*,I_*)-\partial_{\bar{\theta}}\bar{f}(\bar{\theta}_*,I_*)|=0.
\end{equation}
From $(G1)$, we have
\[ |\partial_{\bar{\theta}}\bar{f}_*(\bar{\theta}_*)|=|\partial_{\bar{\theta}}\bar{f}(\bar{\theta}_*,I_*)|:=3\zeta>0. \]
Hence, by~\eqref{speed1}, for $\varepsilon$ small enough with respect to $\zeta$, we have
\begin{equation*}
|\partial_{\bar{\theta}}\bar{f}_\varepsilon(\bar{\theta}_*,I_*)|\geq 2\zeta>0.
\end{equation*}
If we let $B_\zeta(I_*)$ be the ball of radius $\zeta$ around $I_*$, for all $I \in B_\zeta(I_*)$ this last inequality gives
\begin{equation*}\label{lowerbound}
|\partial_{\bar{\theta}}\bar{f}(\bar{\theta}_*,I)|\geq\zeta>0.
\end{equation*}
Up to taking $\zeta$ smaller with respect to the distance of $I_*$ to $\partial B=\bar{B}\setminus B$, we may assume that $B_\zeta(I_*)$ is contained in $B$, and then taking $\varepsilon$ smaller, we may assume that $B_\zeta(I_*)$ is contained in $B'=B-C\mu_{\tilde{\alpha}}(\varepsilon)$. From now on, to simplify notations, when convenient we will suppress all constants depending on $n$, $k$, $\zeta$ and the function $\Delta_{\tilde{\alpha}}$ by using the symbols $\lesssim$, $\gtrsim$ and $\sim$.

We start by assuming $\varepsilon \lesssim 1$ so that~\eqref{mu} is satisfied and Proposition~\ref{normal1} can be applied. Let $H \circ \Phi$ be the Hamiltonian in formal form, which is defined on $\T^n \times B'$, and has the form
\[ \bar{H}(\theta,I)=\alpha\cdot I+\varepsilon \bar{f}_\varepsilon(\bar{\theta},I)+ \varepsilon f_\varepsilon^+(\theta,I)=\tilde{\alpha}\cdot \tilde{I}+\varepsilon \bar{f}_\varepsilon(\bar{\theta},I)+ \varepsilon f_\varepsilon^+(\theta,I) \]
where $I=(\bar{I},\tilde{I}) \in \R^d \times \R^{n-d}$. Consider the solution $(\theta(t),I(t))$ of the system associated to $H \circ \Phi$ with $\bar{\theta}(0)=\bar{\theta}_* \in \T^d$, $\tilde{\theta}(0) \in \T^{n-d}$ arbitrary and $I(0) \in B_\zeta(I_*)$. It satisfies the following system of equation:
\begin{equation}\label{sys}
\begin{cases}
\dot{\bar{\theta}}(t)=\varepsilon \partial_{\bar{I}}\bar{f}_\varepsilon(\bar{\theta}(t),I(t))+\varepsilon\partial_{\bar{I}}f_\varepsilon^+(\theta(t),I(t)), \\
\dot{\tilde{\theta}}(t)=\tilde{\alpha}+\varepsilon \partial_{\tilde{I}}\bar{f}_\varepsilon(\bar{\theta}(t),I(t))+\varepsilon\partial_{\tilde{I}}f_\varepsilon^+(\theta(t),I(t)), \\
\dot{\bar{I}}(t)=-\varepsilon \partial_{\bar{\theta}}\bar{f}_\varepsilon(\bar{\theta}(t),I(t))-\varepsilon\partial_{\bar{\theta}} f_\varepsilon^+(\theta(t),I(t)), \\
\dot{\tilde{I}}(t)=-\varepsilon\partial_{\tilde{\theta}} f_\varepsilon^+(\theta(t),I(t)).
\end{cases}
\end{equation}
Using the fact that $\bar{f}_\varepsilon$ has unit norm and the estimates~\eqref{estimates} on $f_\varepsilon^+$, we obtain from the first, third and fourth equation of~\eqref{sys} that, for $t \leq \tau\sim \varepsilon^{-1}$ and $\varepsilon \lesssim 1$,
\[ |\bar{\theta}(t)-\bar{\theta}(0)| \leq \tau\left|\dot{\bar{\theta}}(t)\right| \lesssim 1+\mu(\varepsilon) \lesssim 1 \]
\[ |\bar{I}(t)-\bar{I}(0)| \leq \tau\left|\dot{\bar{I}}(t)\right| \lesssim 1+\mu(\varepsilon) \lesssim 1 \]
\[ |\tilde{I}(t)-\tilde{I}(0)| \leq \tau\left|\dot{\tilde{I}}(t)\right| \lesssim \mu(\varepsilon) \lesssim 1. \]
These inequalities, together with~\eqref{lowerbound} and our choices of $\bar{\theta}(0)=\bar{\theta}_*$ and $I(0)=(\bar{I}(0),\tilde{I}(0)) \in B_\zeta(I_*)$, implies that for $t \leq \tau$,
\begin{equation}\label{lowerbound2}
|\partial_{\bar{\theta}}\bar{f}_\varepsilon(\bar{\theta}_*(t),I(t))|\gtrsim 1.
\end{equation}
But then using again the third equation of~\eqref{sys}, one has
\[ \left|\dot{\bar{I}}(t)\right| \gtrsim \varepsilon -\varepsilon\mu(\varepsilon) \gtrsim \varepsilon \]
which eventually gives
\[ |\bar{I}(\tau)-\bar{I}(0)|\gtrsim \tau\varepsilon \sim 1 \]
and so
\[ |I(\tau)-I(0)|\gtrsim 1. \]
This proves the theorem for $H \circ \Phi$, but the statement for $H$ is then obvious. Indeed, recall that $\Phi$ is symplectic, so that if $(\theta(t),I(t))$ is a solution of the system associated to $H \circ \Phi$, then $(\tilde{\theta}(t),\tilde{I}(t)):=\Phi(\theta(t),I(t))$ is a solution of the system associated to $H$. Moreover, $\Phi$ is close to the identity as expressed in~\eqref{estimates}, so that in particular the image of $\T^n \times B_\zeta(I^*)$ by $\Phi$ contains $\T^n \times B_*$, where $B_*$ is some open ball in $B$. Consider an arbitrary solution $(\tilde{\theta}(t),\tilde{I}(t))$ with $\tilde{I}(0) \in B_*$, then the corresponding solution $(\theta(t),I(t))$ has $I(0) \in B_\zeta(I^*)$, and therefore by the preceding, 
\[ |I(\tau)-I(0)|\gtrsim 1. \]
But then
\[ |\tilde{I}(\tau)-\tilde{I}(0)| \geq |I(\tau)-I(0)|-|I(\tau)-\tilde{I}(\tau)|-|I(0)-\tilde{I}(0)| \]
and using the estimate~\eqref{estimates} on $\Phi$, we have, for $\varepsilon$ small enough, 
\[ |\tilde{I}(\tau)-\tilde{I}(0)| \geq |I(\tau)-I(0)|-2C\mu_{\tilde{\alpha}}(\varepsilon)\gtrsim 1.  \]
This concludes the proof. 
\end{proof}

\begin{proof}[Proof of Theorem 1']
Now we are assuming that $f$ satisfies $(G1')$, that is for all $I \in \bar{B}$, there exists $\bar{\theta}=\bar{\theta}(I) \in \T^d$ such that $\partial_{\bar{\theta}}\bar{f}(\bar{\theta},I) \neq 0$. We are also assuming
\begin{equation}\label{speed11}
\lim_{\varepsilon \rightarrow 0}|\bar{f}_\varepsilon-\bar{f}|_{C^1(\T^d \times \bar{B})}=0.
\end{equation}
Given any $I \in \bar{B}$, there exist a positive constant $\zeta_I$ and an open neighborhood $U(I)$ of $I$ in $\bar{B}$ such that for any $I' \in U(I)$, we have 
\[ |\partial_{\bar{\theta}}\bar{f}(\bar{\theta},I')| \geq \zeta_I \] 
for some $\bar{\theta}=\bar{\theta}(I) \in \T^d$. The sets $U(I)$, for $I \in \bar{B}$, forms an open cover of $\bar{B}$ from which one can extract a finite open cover. Hence there exists $\zeta>0$ such that for all $I\in \bar{B}$, we have 
\[ |\partial_{\bar{\theta}}\bar{f}(\bar{\theta},I)| \geq \zeta \] 
for some $\bar{\theta}=\bar{\theta}(I) \in \T^d$. Using~\eqref{speed11}, for $\varepsilon$ sufficiently small with respect to $\zeta$, we obtain
\[ |\partial_{\bar{\theta}}\bar{f}_\varepsilon(\bar{\theta},I)| \geq \zeta. \] 
Once we have obtained a uniform constant, the rest of the proof is now similar to the proof of Theorem~\ref{thm1}.
\end{proof}

\section{Proof of Theorem~\ref{thm2} and Theorem 2'}\label{s5}

The proof of Theorem~\ref{thm2} will follow easily from Proposition~\ref{normal2} and a version of the classical KAM theorem that we will recall below.

Indeed, as before, Proposition~\ref{normal2} provides us with a (symplectic, close to the identity) change of coordinates which shows that the dominant part of the Hamiltonian is given by $h+\varepsilon\bar{f}_\varepsilon+\varepsilon g_\varepsilon$, where $\bar{f}_\varepsilon$ is the (full) average of $f_\varepsilon$. The system associated to this dominant part is now integrable, and because of the condition $(G2)$, the system associated to  $h+\varepsilon\bar{f}_\varepsilon$ is non-degenerate in the sense of Kolmogorov, and so is the system associated to $h+\varepsilon\bar{f}_\varepsilon+\varepsilon g_\varepsilon$ as $\varepsilon g_\varepsilon$ can be considered as an (integrable) perturbation. One has to be careful because the non-degeneracy is $\varepsilon$-dependent; yet we will see that this causes no problem.  
 
Let us now recall the statement of the classical KAM theorem. Consider a domain $D \subset \R^n$, two positive constants $\varrho$ and $\sigma$ and a Hamiltonian $H$ real-analytic on $V_\sigma \T^n \times V_\varrho D$ of the form
\begin{equation}\label{HamP}\tag{$H$}
H(\theta,I)=H_0(I)+H_1(\theta,I), \quad (\theta,I) \in V_\sigma \T^n \times V_\varrho D. 
\end{equation}
For any $I \in V_\varrho D$, the Hessian $\nabla^2 H_0(I)$ is assumed to be non-singular and moreover, we assume the existence of a constant $M>0$ such that
\begin{equation}\label{hess}
|\nabla^2 H_0(I)| \leq M, \quad |\left(\nabla^2 H_0(I)\right)^{-1}| \leq M.
\end{equation}
We can now state a version of the classical KAM theorem.

\begin{theorem}[KAM]\label{thmK}
Let $H$ be as in~\eqref{HamP}, with $H_0$ satisfying~\eqref{hess}. Assume moreover that $\nabla H_0$ is a diffeomorphism of $V_\varrho D$ onto its image. Then there exist constants $\nu_0>0$ and $E>0$, which depend only on $n$, $\sigma$, $\varrho$ and $M$, such that if
\[ \nu:=|H_1|_{\sigma,\varrho}\leq \nu_0,  \]
there exists a set $\mathcal{K}_\nu \subset \T^n \times D$ which consists of real-analytic Lagrangian quasi-periodic tori invariant by the Hamiltonian flow of $H$. Moreover, we have the measure estimate
\[ \mathrm{Leb}(\T^n \times D \setminus \mathcal{K}_\nu) \leq E\sqrt{\nu}\mathrm{Leb}(\T^n \times D) \] 
where $\mathrm{Leb}$ denotes Lebesgue measure. 
\end{theorem}

This is exactly the content of Theorem A in~\cite{Pos82}, to which we refer for a proof. We can now complete the proof of Theorem~\ref{thm2}.

\begin{proof}[Proof of Theorem~\ref{thm2}]
Recall that we are considering $H$ as in~\eqref{Ham2}, with $f$ satisfying $(G2)$, which is the existence of $I_* \in B$ such that $\nabla^2 \bar{f}(I_*)$ is non-singular.
Recall also that
\begin{equation}\label{speed2}
\lim_{\varepsilon \rightarrow 0}|\nabla^2 \bar{f}_\varepsilon(I_*)-\nabla^2 \bar{f}(I_*)|=0.
\end{equation}
From $(G2)$ we know that there exists a positive constant $M>0$ such that
\begin{equation*}
|\nabla^2 f(I^*)| \leq M/2, \quad |\left(\nabla^2 f(I^*)\right)^{-1}| \leq M/2.
\end{equation*}
Then, because of~\eqref{speed2}, for $\varepsilon$ sufficiently small with respect to $M$, we obtain
\begin{equation}\label{nonder}
|\nabla^2 f_\varepsilon(I^*)| \leq M, \quad |\left(\nabla^2 f_\varepsilon(I^*)\right)^{-1}| \leq M.
\end{equation}
As in the proof of Theorem~\ref{thm1}, we will now suppress all constants depending on $n$, $s$, $r$, $M$, the distance of $I_*$ to $\partial B=\bar{B}\setminus B$ and the function $\Delta_\alpha$ by using the symbols $\lesssim$, $\gtrsim$ and $\sim$.

First, we assume that $\varepsilon \lesssim 1$ so that $\mu_{\alpha}(\varepsilon) \leq \mu_0$ and Proposition~\ref{normal2} applies. Let us first prove the statement for the Hamiltonian $\bar{H}:=H \circ \Phi$ in normal form. To prove such a statement, let us consider the rescaled Hamiltonian $\bar{H}_\varepsilon:=\varepsilon^{-1}\bar{H}$, that is
\[ \bar{H}_\varepsilon(\theta,I)=\varepsilon^{-1}\alpha\cdot I+\bar{f}_\varepsilon(I)+g_\varepsilon(I)+f_\varepsilon^+(\theta,I) \]
that we can write
\[ \bar{H}_\varepsilon(\theta,I)=H_0(I)+H_1(\theta,I) \]
with
\[ H_0(I)=\varepsilon^{-1}\alpha\cdot I+\bar{f}_\varepsilon(I)+g_\varepsilon(I), \quad H_1(\theta,I)=f_\varepsilon^+(\theta,I). \]
We have
\[\nabla^2 H_0(I_*)=\nabla^2 \bar{f}_\varepsilon(I_*)+\nabla^2 g_\varepsilon(I_*), \]
so that using the estimate~\eqref{estimates2} on $g_\varepsilon$ and Cauchy estimates, we obtain
\[ |\nabla^2 H_0(I_*)-\nabla^2 \bar{f}_\varepsilon(I_*)| \lesssim \mu_{\alpha}(\varepsilon). \]
For $\varepsilon \lesssim 1$, using this last estimate, and the inequality~\eqref{nonder}, we obtain
\begin{equation*}
|\nabla^2 H_0(I_*)| \lesssim 1, \quad |\left(\nabla^2 H_0(I_*)\right)^{-1}| \lesssim 1.
\end{equation*}
Hence there exists $\zeta \sim 1$ such that $\nabla^2 H_0(I)$ is non-singular for all $I \in B_\zeta(I_*)$. Restricting $\zeta$ if necessary, we may assume that $\nabla^2 H_0(I)$ is non-singular for all $I \in V_\varrho B_\zeta(I_*)$, for some $0 < \varrho \leq r/2$, with the estimates
\begin{equation*}
|\nabla^2 H_0(I)| \lesssim 1, \quad |\left(\nabla^2 H_0(I)\right)^{-1}| \lesssim 1.
\end{equation*}
Moreover, by a further restriction of $\zeta$ if necessary, $\nabla H_0$ can be assumed to be a diffeomorphism of $V_\varrho B_\zeta(I_*)$ onto its image. Now we set $\sigma=s/2$, and the estimate~\eqref{estimates2} on $f_\varepsilon^+$ yields
\[ |H_1|_{\sigma,\varrho} \leq \nu_\alpha(\varepsilon) \sim \exp\left(-c'\mu_{\alpha}(\varepsilon)^{-1}\right). \]
Assuming again $\varepsilon \lesssim 1$ so that $\nu_\alpha(\varepsilon) \leq \nu_0$, we can apply Theorem~\ref{thmK} with $D=B_\zeta(I_*)$ and $\nu=\nu_\alpha(\varepsilon)$: there exists a set $\mathcal{K}_{\nu_\alpha(\varepsilon)} \subset \T^n \times B_\zeta(I_*)$ which consists of real-analytic Lagrangian quasi-periodic tori invariant by the Hamiltonian flow of $\bar{H}_\varepsilon$. Moreover, we have the measure estimate
\[ \mathrm{Leb}(\T^n \times B_\zeta(I_*) \setminus \mathcal{K}_{\nu_\alpha(\varepsilon)}) \lesssim \sqrt{\nu_\alpha(\varepsilon)}\mathrm{Leb}(\T^n \times B_\zeta(I_*)) \sim \exp\left(-\bar{c}\mu_{\alpha}(\varepsilon)^{-1}\right)\mathrm{Leb}(\T^n \times B_\zeta(I_*)) \]
with $\bar{c}=c'/2$. Now observe that the Hamiltonian flow of $\tilde{H}_\varepsilon$ differs from the Hamiltonian flow of $\tilde{H}$ by a constant time change: more precisely, $(\theta(t),I(t))$ is a solution of the system associated to $\bar{H}_\varepsilon$ if and only if $(\theta(\varepsilon^{-1}t),I(\varepsilon^{-1}t))$ is a solution of the system associated to $\bar{H}$. Hence $\mathcal{K}_{\nu_\alpha(\varepsilon)} \subset \T^n \times B_\zeta(I_*)$ is still invariant by the Hamiltonian flow of $\bar{H}$. To conclude, let $\mathcal{K}_\varepsilon:=\Phi^{-1}(\mathcal{K}_{\nu_\alpha(\varepsilon)})$ and set $B_*=B_{\zeta/2}(I_*)$. Using the estimate~\eqref{estimates2} on $\Phi$, for $\varepsilon \lesssim 1$ one can ensure that $\mathcal{K}_\varepsilon \subset \T^n \times B_*$. Moreover, since $\Phi$ is real-analytic and symplectic, then $\mathcal{K}_\varepsilon \subset \T^n \times B_*$ consists of real-analytic Lagrangian quasi-periodic tori invariant by the Hamiltonian flow of $H$. Moreover, using again the estimate~\eqref{estimates2} on $\Phi$, for $\varepsilon \lesssim 1$ the Jacobian of $\Phi$ can be made arbitrarily close to one, and since 
\[ \mathrm{Leb}(\T^n \times B_*) \sim \mathrm{Leb}(\T^n \times B_\zeta(I_*)) \]   
we eventually obtain
\[ \mathrm{Leb}(\T^n \times B_* \setminus \mathcal{K}_\varepsilon) \lesssim \exp\left(-\bar{c}\mu_{\alpha}(\varepsilon)^{-1}\right)\mathrm{Leb}(\T^n \times B_*). \]
This concludes the proof.      
\end{proof}

\begin{proof}[Proof of Theorem 2']
Now we are assuming that $f$ satisfies $(G2')$, that is $\nabla \bar{f} : \bar{B} \rightarrow \R^n$ is a diffeomorphism onto its image, and 
\begin{equation}\label{speed22}
\lim_{\varepsilon \rightarrow 0}|\bar{f}_\varepsilon-\bar{f}|_{C^2(\bar{B})}=0.
\end{equation}
Given any $I \in \bar{B}$, there exist a positive constant $M_I$ and an open neighborhood $U(I)$ of $I$ in $\bar{B}$ such that for any $I' \in U(I)$, we have 
\begin{equation*}
|\nabla^2 \bar{f}(I')| \leq M_I/2, \quad |\left(\nabla^2 \bar{f}(I')\right)^{-1}| \leq M_I/2.
\end{equation*}
As in the proof of Theorem 1', by a compactness argument one can find a positive constant $M$ such that for all $I \in \bar{B}$,
\begin{equation*}
|\nabla^2 \bar{f}(I)| \leq M/2, \quad |\left(\nabla^2 \bar{f}(I)\right)^{-1}| \leq M/2.
\end{equation*} 
Using~\eqref{speed22}, for $\varepsilon$ sufficiently small with respect to $M$, we get
\begin{equation*}
|\nabla^2 \bar{f}_\varepsilon(I)| \leq M, \quad |\left(\nabla^2 \bar{f}_\varepsilon(I)\right)^{-1}| \leq M.
\end{equation*} 
Having obtained a uniform constant, the rest of the proof is now similar to the proof of Theorem~\ref{thm2}.
\end{proof}

\section{Proof of Theorem~\ref{thm3} and Theorem 3'}\label{s6}

The proof of Theorem~\ref{thm3} will follow from Proposition~\ref{normal2} and a version of the classical Nekhoroshev theorem that we will recall below. 

Once again, Proposition~\ref{normal2} gives us a (symplectic, close to the identity) change of coordinates which shows that the dominant part of the Hamiltonian, which is given by $h+\varepsilon\bar{f}_\varepsilon+\varepsilon g_\varepsilon$, is integrable. Using the condition $(G3)$, we will show that the system associated to $h+\varepsilon\bar{f}_\varepsilon$ is generic in the sense of Nekhoroshev, and so is the system associated to $h+\varepsilon\bar{f}_\varepsilon+\varepsilon g_\varepsilon$. Here also, we will have to show that the $\varepsilon$-dependence of the integrable part is not an obstacle. 

The Nekhoroshev theorem requires a condition of steepness on the integrable part of the Hamiltonian, so we first recall this definition.

\begin{definition}\label{funcsteep}  
A differentiable function $h : B \rightarrow \R$ is steep on a domain $B' \subseteq B$ if there exist positive constants $C',\delta',p_l$, for any integer $l\in [1,n-1]$, and $\kappa$ such that for all $I \in B'$, we have $|\nabla h(I)| \geq \kappa$ and, for all integer $l\in [1,n-1]$, for all vector space $\Lambda \in \R^n$ of dimension $l$, letting $\lambda=I+\Lambda$ the associated affine subspace passing through $I$ and $h_\lambda$ the restriction of $h$ to $\lambda$, the inequality
\[ \max_{0 \leq \eta \leq \xi}\;\min_{|I'-I|=\eta, \; I' \in \lambda}|\nabla h_\lambda(I')-\nabla h_\lambda(I)|>C'\xi^{p_l} \]
holds true for all $0 < \xi \leq\delta'$. We say that $h$ is $(\kappa,C',\delta',(p_l)_{l=1,\ldots,n-1})$-steep on $B'$ and, if all the $p_i=p$, we say that $h$ is $(\kappa,C',\delta',p)$-steep on $B'$.
\end{definition} 

Next we need to make a link between this definition of steep function and the definition of stably steep polynomials (Definition~\ref{stabsteep} in Section~\ref{s2}): a function whose Taylor expansion at some point is stably steep is steep on some neighborhood of this point. Here's a precise statement.

\begin{proposition}\label{propsteep1}
Let $h : B_{\bar{\zeta}}(I_*) \rightarrow \R$ be a function of class $C^{m_0+1}$ such that $|\nabla h(I_*)|\geq\varpi>0$ and such that $T_{I_*}^{m}h$ is $(\rho,C,\delta)$-stably steep. Then for $\zeta>0$ sufficiently small with respect to $\bar{\zeta}$, $|h|_{C^{m_0+1}(B_{\bar{\zeta}})}$, $\rho$, $\varpi$, $m_0$, $C$ and $\delta$, the function $h$ is $(\kappa,C',\delta',m-1)$-steep on $B_{\zeta}(I_*)$ with
\[ \kappa=\varpi/2, \quad C'=C/2, \quad \delta'=\zeta.  \] 
\end{proposition}

For a proof, we refer to Theorem $2.2$ in~\cite{BFN15} or Proposition $3.2$ in~\cite{BFN15b}.

We can eventually state the theorem of Nekhoroshev (\cite{Nek77}, \cite{Nek79}). To do this, we consider again a real-analytic Hamiltonian $H$ as in~\eqref{HamP}, that is
\begin{equation*}
H(\theta,I)=H_0(I)+H_1(\theta,I), \quad (\theta,I) \in V_\sigma \T^n \times V_\varrho D,
\end{equation*}
where $D$ is some domain in $\R^n$ and $\sigma$ and $\varrho$ are positive constants. The integrable Hamiltonian will be assumed to be a steep function on $D$, and its Hessian will be assumed to be uniformly bounded on $V_\varrho D$, that is there exists a positive constant $M$ such that for all $I \in V_\varrho D$,
\begin{equation}\label{hess2}
|\nabla^2 H_0(I)| \leq M.
\end{equation}
Here's the statement. 

\begin{theorem}[Nekhoroshev]\label{thmN}
Let $H$ be as in~\eqref{HamP}, with $H_0$ a $(\kappa,C',\delta',p)$-steep function on $D$ satisfying~\eqref{hess2}. Then there exists a constant $\nu_0>0$, which depends only on $n$, $\sigma$, $\varrho$, $\kappa$, $C'$, $\delta'$, $p$ and $M$, and constants $a>0$ and $b>0$ which depend only on $n$ and $p$, such that if
\[ \nu:=|H_1|_{\sigma,\varrho}\leq \nu_0,  \]
then for any solution $(\theta(t),I(t))$ with $I(0)\in D-\nu^b$, we have
\[ |I(t)-I(0)| \leq \nu^b/2, \quad |t| \leq \nu^{-1}\exp\left(\nu^{-a}\right).  \]
\end{theorem}

This is exactly the content of the main theorem in~\cite{Nek77}, Section \S 4.4, to which we refer for a proof. We can now complete the proof of Theorem~\ref{thm3}.

\begin{proof}[Proof of Theorem~\ref{thm3}]
Recall that we are considering $H$ as in~\eqref{Ham2}, with $f$ satisfying $(G3)$, which is the existence of a point $I_* \in B$ and an integer $m \geq m_n$ such that the polynomial $T_2^{m}\bar{f}(I_*) \in P_2(n,m)$ is stably steep. Recall also that
\begin{equation}\label{speed3}
\lim_{\varepsilon \rightarrow 0}\left(\max_{2 \leq l \leq m}|\nabla^l \bar{f}_\varepsilon(I_*)-\nabla^l \bar{f}(I_*)|\right)=0.
\end{equation}
From $(G3)$ we know that there exist positive constant $\rho$, $C$ and $\delta$ such that the polynomial $T_2^{m}\bar{f}(I_*) \in P_2(n,m)$ is $(2\rho,C,\delta)$-stably steep. Then, using~\eqref{speed3}, for $\varepsilon$ sufficiently small with respect to $\delta$, the polynomial $T_2^{m}\bar{f}_\varepsilon(I_*) \in P_2(n,m)$ is $(\rho,C,\delta)$-stably steep. As in the proofs of Theorem~\ref{thm1} and Theorem~\ref{thm2}, we will now suppress all constants depending on $n$, $s$, $r$, $m$, $\rho$, $C$, $\delta$, $|\nabla^2 \bar{f}(I_*)|$, the distance of $I_*$ to $\partial B=\bar{B}\setminus B$ and the function $\Delta_\alpha$ by using the symbols $\lesssim$, $\gtrsim$ and $\sim$. 

First, we assume that $\varepsilon \lesssim 1$ so that $\mu_{\alpha}(\varepsilon) \leq \mu_0$ and Proposition~\ref{normal2} applies. We will first prove a statement for the Hamiltonian $\bar{H}:=H \circ \Phi$ in normal form. To prove such a statement, again, we consider the rescaled Hamiltonian $\bar{H}_\varepsilon:=\varepsilon^{-1}\bar{H}$, that can be written 
\[ \bar{H}_\varepsilon(\theta,I)=H_0(I)+H_1(\theta,I) \]
with
\[ H_0(I)=\varepsilon^{-1}\alpha\cdot I+\bar{f}_\varepsilon(I)+g_\varepsilon(I), \quad H_1(\theta,I)=f_\varepsilon^+(\theta,I). \]
Let us first establish the fact that $H_0$ is steep in a neighborhood of the point $I_*$. Using the estimate~\eqref{estimates2} on $g_\varepsilon$ and Cauchy estimates, we obtain that for all $2 \leq l \leq m$,
\[ |\nabla^l H_0(I_*)-\nabla^l \bar{f}_\varepsilon(I_*)| \lesssim \mu_{\alpha}(\varepsilon). \]
Therefore for $\varepsilon \lesssim 1$, the polynomial $T_2^{m}H_0(I_*) \in P_2(n,m)$ is stably steep. Moreover, it is easy to see that for $\varepsilon \lesssim 1$, we have
\[ |\nabla H_0(I_*)| \gtrsim \varepsilon^{-1} \gtrsim 1. \] 
We can thus apply Proposition~\ref{propsteep1} to find $\zeta \sim 1$ such that $H_0$ is a steep function on $B_\zeta(I_*)$. Choosing $\varrho=r/2$ and $\sigma=s/2$, we also have
\[\sup_{I \in V_\varrho B_\zeta(I_*)}|\nabla^2 H_0(I)| \lesssim 1\] 
and
\[ |H_1|_{\sigma,\varrho} \leq \nu_\alpha(\varepsilon) \sim \exp\left(-c'\mu_{\alpha}(\varepsilon)^{-1}\right). \]
Assuming again $\varepsilon \lesssim 1$ so that $\nu_\alpha(\varepsilon) \leq \nu_0$, we can eventually apply Theorem~\ref{thmN} with $D=B_\zeta(I_*)$ and $\nu=\nu_\alpha(\varepsilon)$ and we obtain the following statement: for any solution $(\theta(t),I(t))$ of the system associated to $\bar{H}_\varepsilon$ with $I(0) \in B_\zeta(I_*)-\nu_\alpha(\varepsilon)^b$, we have
\[ |I(t)-I(0)| \leq \nu_\alpha(\varepsilon)^b/2, \quad |t| \leq \nu_\alpha(\varepsilon)^{-1}\exp\left(\nu_\alpha(\varepsilon)^{-a}\right).  \]
For $\varepsilon \lesssim 1$, observe that we can assume that $B_\zeta(I_*)-\nu_\alpha(\varepsilon)^b$ contains $B_{\zeta/2}(I_*)$, hence the estimate above applies for any $I(0) \in B_{\zeta/2}(I_*)$. Now let us recall that $(\theta(t),I(t))$ is a solution of the system associated to $\bar{H}_\varepsilon$ if and only if $(\theta(\varepsilon^{-1}t),I(\varepsilon^{-1}t))$ is a solution of the system associated to $\bar{H}$. Hence for any solution $(\theta(t),I(t))$ of the system associated to $\bar{H}$ with $I(0) \in B_{\zeta/2}(I_*)$, we have
\[ |I(t)-I(0)| \leq \nu_\alpha(\varepsilon)^b/2, \quad |t| \leq \varepsilon^{-1}\nu_\alpha(\varepsilon)^{-1}\exp\left(\nu_\alpha(\varepsilon)^{-a}\right).  \] 
Recalling the definition of $\nu_\alpha(\varepsilon)$, for $\varepsilon$ sufficiently small this inequality can be arranged to give
\[ |I(t)-I(0)| \leq \nu_\alpha(\varepsilon)^b/2, \quad |t| \leq \varepsilon^{-1}\exp\left(\exp\left(-\tilde{c}\mu_{\alpha}(\varepsilon)^{-1}\right)\right)  \]
with $\tilde{c} \sim c'$. Coming back to the original Hamiltonian $H$, as $\Phi$ is symplectic, $(\theta(t),I(t))$ is a solution of the system associated to $H \circ \Phi$ if and only if $(\tilde{\theta}(t),\tilde{I}(t)):=\Phi(\theta(t),I(t))$ is a solution of the system associated to $H$. Since $\Phi$ is closed to the identity as expressed in~\eqref{estimates2}, the image of $\T^n \times B_{\zeta/2}(I_*)$ contains $\T^n \times B_*$ for some open ball $B_*$ contained in $B$. Consider an arbitrary solution $(\tilde{\theta}(t),\tilde{I}(t))$ with $\tilde{I}(0) \in B_*$, then the corresponding solution $(\theta(t),I(t))$ has $I(0) \in B_{\zeta/2}(I_*)$, and therefore  
\[ |I(t)-I(0)| \leq \nu_\alpha(\varepsilon)^b/2, \quad |t| \leq \varepsilon^{-1}\exp\left(\exp\left(-\tilde{c}\mu_{\alpha}(\varepsilon)^{-1}\right)\right). \] 
But then, for $|t| \leq \varepsilon^{-1}\exp\left(\exp\left(-\tilde{c}\mu_{\alpha}(\varepsilon)^{-1}\right)\right)$, we have
\[ |\tilde{I}(t)-\tilde{I}(0)| \leq |\tilde{I}(t)-I(t)|+|I(t)-I(0)|+|I(0)-\tilde{I}(0)| \]
and using the estimate~\eqref{estimates2} on $\Phi$, we have, for $|t| \leq \varepsilon^{-1}\exp\left(\exp\left(-\tilde{c}\mu_{\alpha}(\varepsilon)^{-1}\right)\right)$, 
\[ |\tilde{I}(\tau)-\tilde{I}(0)| \leq \nu_\alpha(\varepsilon)^b/2+2C\mu_\alpha(\varepsilon) \lesssim \mu_\alpha(\varepsilon).  \]
This concludes the proof. 
\end{proof}

\begin{proof}[Proof of Theorem 3']
Now we are assuming that $f$ satisfies $(G3')$, that is there exists an integer $m \geq m_n$ such that for all $I \in \bar{B}$, the polynomial $T_2^{m}\bar{f}(I) \in P_2(n,m)$ is stably steep. We are also assuming 
\begin{equation}\label{speed33}
\lim_{\varepsilon \rightarrow 0}|\bar{f}_\varepsilon-\bar{f}|_{C^m(\bar{B})}=0.
\end{equation}
As in the proofs of Theorem 1' and Theorem 2', since the set of stably steep polynomials is open, by a compactness argument one can find positive constants $\rho$, $C$ and $\delta$ such that for all $I \in \bar{B}$, the polynomial $T_2^{m}\bar{f}(I) \in P_2(n,m)$ is $(2\rho,C,\delta)$-stably steep. Then, using~\eqref{speed33}, for $\varepsilon$ sufficiently small with respect to $\delta$, for all $I \in \bar{B}$ the polynomial $T_2^{m}\bar{f}_\varepsilon(I) \in P_2(n,m)$ is $(\rho,C,\delta)$-stably steep. As before, once we have obtained uniform constants, the rest of the proof is now similar to the proof of Theorem~\ref{thm3}.
\end{proof}

\addcontentsline{toc}{section}{References}
\bibliographystyle{amsalpha}
\bibliography{PertLineaire}

\providecommand{\bysame}{\leavevmode\hbox to3em{\hrulefill}\thinspace}
\providecommand{\MR}{\relax\ifhmode\unskip\space\fi MR }
\providecommand{\MRhref}[2]{%
  \href{http://www.ams.org/mathscinet-getitem?mr=#1}{#2}
}
\providecommand{\href}[2]{#2}
\begin{thebibliography}{BFN15b}

\bibitem[Arn63a]{Arn63a}
V.I. Arnol'd, \emph{{Proof of a theorem of {A}.{N}. {K}olmogorov on the
  invariance of quasi-periodic motions under small perturbations}}, Russ. Math.
  Surv. \textbf{18} (1963), no.~5, 9--36.

\bibitem[Arn63b]{Arn63b}
\bysame, \emph{{Small denominators and problems of stability of motion in
  classical and celestial mechanics}}, Russ. Math. Surv. \textbf{18} (1963),
  no.~6, 85--191.

\bibitem[Arn64]{Arn64}
\bysame, \emph{Instability of dynamical systems with several degrees of
  freedom}, Sov. Math. Doklady \textbf{5} (1964), 581--585.

\bibitem[BFN15a]{BFN15b}
A.~Bounemoura, B.~Fayad, and L.~Niederman, \emph{Double exponential stability
  for generic invariant tori in {H}amiltonian systems}, Preprint, 2015.

\bibitem[BFN15b]{BFN15}
\bysame, \emph{Double exponential stability for generic real-analytic elliptic
  equilibrium points}, Preprint, 2015.

\bibitem[BK14]{BK14}
A.~Bounemoura and V.~Kaloshin, \emph{Generic fast diffusion for a class of
  non-convex {H}amiltonians with two degrees of freedom}, Moscow Mathematical
  Journal \textbf{14} (2014), no.~2, 181--203.

\bibitem[BK15]{BK15}
\bysame, \emph{A note on micro-instability for {H}amiltonian systems close to
  integrable}, Proceedings of the AMS (2015), To appear.

\bibitem[Bou12]{Bou12}
A.~Bounemoura, \emph{Optimal stability and instability for near-linear
  {H}amiltonians}, Annales Henri Poincaré \textbf{13} (2012), no.~4, 857--868.

\bibitem[Bou13a]{Bou13a}
\bysame, \emph{Normal forms, stability and splitting of invariant manifolds
  {I}. {G}evrey {H}amiltonians}, Regul. Chaotic Dyn. \textbf{18} (2013), no.~3,
  237--260.

\bibitem[Bou13b]{Bou13b}
\bysame, \emph{Normal forms, stability and splitting of invariant manifolds
  {II}. {F}initely differentiable {H}amiltonians}, Regul. Chaotic Dyn.
  \textbf{18} (2013), no.~3, 261--276.

\bibitem[Bou14]{Bou14a}
\bysame, \emph{Non-degenerate {L}iouville tori are {KAM} stable}, Preprint,
  arXiv:1412.0509, 2014.

\bibitem[Bou15]{Bou15mrl}
\bysame, \emph{Instability of resonant invariant tori in {H}amiltonian
  systems}, Mathematical Research Letters (2015), To appear.

\bibitem[CV89]{CV89}
B.V. Chirikov and V.V. Vecheslavov, \emph{How fast is the {A}rnold diffusion},
  Preprint INP 89-72 (1989), Novosibirsk 1989.

\bibitem[Dou88]{Dou88}
R.~Douady, \emph{Stabilité ou instabilité des points fixes elliptiques}, Ann.
  Sci. Ec. Norm. Sup. \textbf{21} (1988), no.~1, 1--46.

\bibitem[Fas90]{Fas90}
F.~Fass\`o, \emph{Lie series method for vector fields and {H}amiltonian
  perturbation theory}, Z. Angew. Math. Phys. \textbf{41} (1990), no.~6,
  843--864 (English).

\bibitem[Kol54]{Kol54}
A.N. Kolmogorov, \emph{On the preservation of conditionally periodic motions
  for a small change in {H}amilton's function}, Dokl. Akad. Nauk. SSSR
  \textbf{98} (1954), 527--530.

\bibitem[MG95]{MG95}
A.~Morbidelli and A.~Giorgilli, \emph{Superexponential stability of {KAM}
  tori}, J. Stat. Phys. \textbf{78} (1995), 1607--1617.

\bibitem[Mos60]{Mos60}
J.~Moser, \emph{On the elimination of the irrationality condition and
  {B}irkhoff's concept of complete stability}, Bol. Soc. Mat. Mexicana (2)
  \textbf{5} (1960), 167--175.

\bibitem[Nek73]{Nek73}
N.N. Nekhoroshev, \emph{Stable lower estimates for smooth mappings and the
  gradients of smooth functions}, Matu. USSR Sbornik \textbf{19} (1973), no.~3,
  425--467.

\bibitem[Nek77]{Nek77}
\bysame, \emph{An exponential estimate of the time of stability of nearly
  integrable {H}amiltonian systems}, Russian Math. Surveys \textbf{32} (1977),
  no.~6, 1--65.

\bibitem[Nek79]{Nek79}
\bysame, \emph{An exponential estimate of the time of stability of nearly
  integrable {H}amiltonian systems {II}}, Trudy Sem. Petrovs \textbf{5} (1979),
  5--50.

\bibitem[Pop04]{Pop04}
G.~Popov, \emph{K{AM} theorem for {G}evrey {H}amiltonians}, Erg. Th. Dyn. Sys.
  \textbf{24} (2004), no.~5, 1753--1786.

\bibitem[Pös82]{Pos82}
J.~Pöschel, \emph{Integrability of {H}amiltonian systems on {C}antor sets},
  Comm. Pure Appl. Math. \textbf{35} (1982), no.~5, 653--696.

\bibitem[Pös01]{Pos01}
\bysame, \emph{A lecture on the classical {KAM} theory}, Katok, Anatole (ed.)
  et al., Smooth ergodic theory and its applications (Seattle, WA, 1999).
  Providence, RI: Amer. Math. Soc. (AMS). Proc. Symp. Pure Math. 69, 707-732,
  2001.

\bibitem[Sev03]{Sev03}
M.~B. Sevryuk, \emph{{The classical KAM theory at the dawn of the Twenty-First
  Century}}, Mosc. Math. J. \textbf{3} (2003), no.~3, 1113--1144.

\end{thebibliography}

\end{document}